\documentclass[11pt,reqno]{amsart}
\usepackage{amsmath,amssymb,amsthm,enumerate,natbib,color,bbm,ifthen,graphicx,overpic}
\usepackage[latin1]{inputenc}

\def\nset{{\mathbb{N}}}
\def\rset{\mathbb R}

\def\eqsp{\;}

\newcommand{\ie}{{\em i.e.} }
\newcommand{\as}{\text{a.s.} }

\newcommand{\un}{\ensuremath{\mathbbm{1}}}
\newcommand{\eqdef}{:=}

\def\Xset{\mathsf{X}} 
\def\Y{\mathcal{Y}} 

\def\PP{\mathbb{P}} 
\def\PE{\mathbb{E}} 
\def\bPP{\overline{\mathbb{P}}} 
\def\bPE{\overline{\mathbb{E}}} 
\def\iPP{\check{\mathbb{P}}} 
\def\iPE{\check{\mathbb{E}}} 

\def\cP{\check{P}} 
\def\cPP{\check{\mathbb{P}}} 
\def\cPE{\check{\mathbb{E}}} 

\def\Id{\mathrm{Id}} %
\def\tv{\mathrm{TV}}
\def\Cset{\mathcal{C}} 
\def\Dset{\mathcal{D}} 
\def\Aset{\mathcal{A}} 

\newcommand{\iid}{\ensuremath{\stackrel{\footnotesize{\textup{i.i.d.}}}{\sim}}}

\newlength{\noteWidth}
\newtheorem{theo}{Theorem}[section]
\newtheorem{lemma}[theo]{Lemma}
\newtheorem{coro}[theo]{Corollary}
\newtheorem{prop}[theo]{Proposition}
\newtheorem{defi}[theo]{Definition}
\theoremstyle{remark}
\newtheorem{rem}{Remark}

\newcounter{hypoconbis}
\newcounter{saveconbis}
\newcommand\debutR{\begin{list} {\textbf{R\arabic{hypoconbis}}}{\usecounter{hypoconbis}}\setcounter{hypoconbis}{\value{saveconbis}}}
\newcommand\finR{\end{list}\setcounter{saveconbis}{\value{hypoconbis}}}

\newcounter{hypocom}
\newcounter{savecom}
\newcommand{\debutA}{\begin{list}{\textbf{A\arabic{hypocom}}}{\usecounter{hypocom}}\setcounter{hypocom}{\value{savecom}}}
\newcommand{\finA}{\end{list}\setcounter{savecom}{\value{hypocom}}}

\newcounter{hypoDM}
\newcounter{saveDM}
\newcommand\debutDM{\begin{list} {\textbf{DM\arabic{hypoDM}}}{\usecounter{hypoDM}}\setcounter{hypoDM}{\value{saveDM}}}
\newcommand\finDM{\end{list}\setcounter{saveDM}{\value{hypoDM}}}

\begin{document}

\title[]{State-dependent Foster-Lyapunov criteria for
  subgeometric convergence of Markov chains}

\author{S. Connor} 
\address{Department of Mathematics, University of York, York, YO10 5DD. United Kingdom}
\email{sbc502@york.ac.uk}

\author{G. Fort}
\address{LTCI, CNRS-TELECOM ParisTech, 46 rue Barrault, 75634 Paris Cédex 13, France}
\email{gfort@tsi.enst.fr}



\begin{abstract}
We consider a form of state-dependent drift condition for a general Markov chain, whereby the chain \emph{subsampled} at some deterministic time satisfies a geometric Foster-Lyapunov condition. We present sufficient criteria for such a drift condition to exist, and use these to partially answer a question posed in \cite{Connor.Kendall-2007} concerning the existence of so-called `tame' Markov chains. Furthermore, we show that our `subsampled drift condition' implies the existence of finite moments for the return time to a small set. 
\end{abstract}

\maketitle


{\bf Acknowledgements} The authors acknowledge the support of CRiSM, University
of Warwick, UK.  This work is also partly supported by the French National
research Agency (ANR) under the program ANR-05-BLAN-0299.  The second author
also thanks Prof. P. Priouret for his many comments and insights on an earlier
version of the manuscript.



\section{Introduction and notation}
\label{sec:intro}

Let $\{\Phi_n, n \geq 0\}$ be a time-homogeneous Markov chain on a state space
$\Xset$, with transition kernel $P$. Our goal in this paper is to develop a new
criterion for determining the ergodic properties of $\Phi$. Specifically, we
consider a form of \emph{state-dependent} drift condition for $\Phi$, whereby
the chain \emph{subsampled} at some deterministic time satisfies a geometric
Foster-Lyapunov condition. Drift conditions are classical tools to prove the
stability of Markov chains. Most of the literature addresses the case when the
drift inequality is satisfied by the kernel
$P$~\cite{meyn:tweedie:1993,jarner:roberts:2002,fort:moulines:2003,douc:fort:moulines:soulier:2004}.
Nevertheless, depending upon the application, it may be far easier to prove a
drift condition for a state-dependent iterated kernel than for $P$ itself (see
e.g.  \cite{meyn:tweedie:1994} and section~\ref{sec:applications-2}).
State-dependent drift conditions for phi-irreducible chains were originally
studied by Meyn and Tweedie~\cite{meyn:tweedie:1994}, who gave criteria for
$\Phi$ to be Harris-recurrent, positive Harris-recurrent and geometrically
ergodic.  The drift condition that we develop in this paper is different from
those in \cite{meyn:tweedie:1994}, as will be highlighted in
Section~\ref{ssec:singleDrift}, and can be used to infer a greater range of
convergence rates for $\Phi$ (including both geometric and subgeometric rates).
When applied to subgeometric rates, the present work extends the theory about
subgeometric chains: in \cite{tuominen:tweedie:1994} (resp.
\cite{douc:fort:moulines:soulier:2004}), it is discussed how a nested family of
drift conditions (resp.  a single drift) on the kernel $P$ is related to the
control of modulated moments of the return-time to some set. In this paper, we
provide similar drift criteria in terms of state-dependent iterates of the
transition kernel. Control of such return-times is a key step to establish
ergodicity and, more generally, limit theorems for the chains. In this paper, we
also address a converse result and show how a state-dependent drift condition can be
deduced from the convergence of the iterates of $P$. Such a converse result
exists for geometric chains (see \cite{meyn:tweedie:1993}) but, to our best
knowledge, this is pioneering work for subgeometric chains.

We begin with a little notation; the unfamiliar reader can refer to
\cite{meyn:tweedie:1993}. For any non-negative function $f$ and $n\in\nset$ we
write $P^nf(x)$ for $\int P^n(x,dy)f(y)$ where $P^n(x,dy)$ denotes the
$n$-step transition probability kernel; and for a signed measure $\mu$ we write
$\mu(f) = \int\mu(dy)f(y)$.  The norm $\|\mu \|_f$ is defined as $\sup_{\{g :
  |g|\leq f \}}|\mu(g)|$. This generalises the total variation norm,
$\|\cdot\|_{\tv} \equiv \|\cdot\|_1$. The first return time to a set $\Aset$ is
denoted by $ \tau_\Aset \eqdef \inf \{n \geq 1, \Phi_n \in \Aset \}$ and the
hitting-time on $\Aset$ is denoted by $\displaystyle{\sigma_\Aset \eqdef \inf
  \{n \geq 0, \Phi_n \in \Aset \}}$.  A set $\Cset$ is called \emph{small} (or
$\nu$-small) if there exist some non-trivial measure $\nu$ and constant
$\varepsilon>0$ such that $ P(x, \cdot) \geq \varepsilon \nu(\cdot)$ for all $
x\in\Cset$.  Any $\sigma$-finite measure $\pi$ satisfying $\pi = \pi P$ is
called \emph{invariant}.  An aperiodic chain $\Phi$ that possesses a finite
invariant measure $\pi$ is \emph{ergodic} and for any $x$, $\| P^n(x,\cdot) -
\pi \|_{\tv} \rightarrow 0$ as $n\to \infty$ (\cite[Theorem
13.0.1]{meyn:tweedie:1993}).  This condition is also equivalent to the existence
of a moment for the return time to some accessible small set $\Cset$:
$\sup_{x\in\Cset} \PE_x [ \tau_\Cset ] <\infty$.  $\Phi$ is said to be
\emph{geometrically ergodic} if there exist a function
$V:\Xset\rightarrow[0,\infty)$ and constants $\gamma\in(0,1)$, $R < \infty$,
such that for any $x \in \Xset$, $ \| P^n(x, \cdot)-\pi(\cdot)\|_{V} \leq R \ 
V(x)\gamma^n$. This is equivalent to the existence of a scale function
$V:\Xset\rightarrow[1,\infty)$, a small set $\Cset$, and constants
$\beta\in(0,1)$, $b<\infty$, such that
\begin{equation}
 \label{eq:FL-drift}
 PV(x) \leq \beta V(x) + b\un_{\Cset} (x) \eqsp, 
\end{equation}
where $\un_\Aset$ is the indicator function of the set $\Aset$. This geometric
drift condition is also equivalent to the existence of an exponential moment
for the return time to the set $\Cset$: $\sup_{x\in\Cset} \PE_x [
\beta^{-\tau_\Cset} ] <\infty$ (\cite[Chapter 15]{meyn:tweedie:1993}).  More
generally, $\Phi$ is said to be \emph{$(f,r)$-ergodic} if there exist functions
$r:\nset\rightarrow [1,\infty)$ and $f:\Xset\rightarrow [1,\infty)$ such that,
for all $x$ in a full and absorbing set, $r(n)\|P^n (x,\cdot)-\pi(\cdot)\|_f
\rightarrow 0$ as $n \to \infty$.

In this paper we will be studying \emph{subgeometrically ergodic} chains; the
class $\Lambda$ of subgeometric rates $r =\{r(n), n\geq 0\}$ is defined in
\cite{nummelin:tuominen:1983} as follows: call $\Lambda_0$ the set of rate
functions $r_0 =\{r_0(n), n\geq 0\}$ such that $r_0(0) \geq 1$, $n \mapsto
r_0(n)$ is nondecreasing and $\lim_{n\to\infty} \ln r_0(n)/n \downarrow 0$.
Then $r \in \Lambda$ iff $r$ is nonnegative, non-decreasing and there exists
$r_0 \in \Lambda_0$ such that $\lim_{n\to\infty} r(n) /r_0(n)=1$. Sufficient
drift conditions for $(f,r)$-ergodicity, relative to the one-step transition
kernel $P$ exist in the literature
\cite{tuominen:tweedie:1994,douc:fort:moulines:soulier:2004}. The converse
result is, to our best knowledge, an open question (except when $f=1$, see
\cite{tuominen:tweedie:1994}).

The remainder of this paper is laid out as follows. In Section~\ref{sec:Tame}
we consider when it is possible to find functions
$n:\Xset\rightarrow\nset_\star$ and $V:\Xset\rightarrow [1,\infty)$, and $\beta
\in (0,1)$ such that
\begin{equation}
  \label{eq:drift-early}
   P^{n(x)}V(x) \leq \beta V(x) + b\un_\Cset(x) \eqsp.
\end{equation}
That is, such that the chain $\Phi$ subsampled at time $n(x)$ exhibits a
geometric drift condition. The sufficient conditions presented are ultimately
based on the existence of moments for the return time to a small set.  In
Section~\ref{sec:Moment} we consider the inverse problem: starting from a drift
condition of the form \eqref{eq:drift-early}, what can be said about the
existence of moments of $\tau_\Cset$?  In Section~\ref{sec:applications} these
results are applied to the classification of tame chains, and to a discrete-time
process forming part of a perfect simulation algorithm for such chains. In
Section~\ref{sec:appli-markov}, it is shown how these results can be used to
prove the subgeometric ergodicity of strong Markov processes.  Proofs of the
main results are provided in Section~\ref{sec:Proof}.

\section{Foster-Lyapunov drift inequalities under subsampling }
\label{sec:Tame}
We first consider when it is possible to deterministically subsample a chain
$\Phi$ at rate $n$, in order to produce a Foster-Lyapunov drift condition, \ie
an inequality of the form $ P^{n(x)}V(x) \leq \beta V(x) + b\un_\Cset(x) $ for
some function $V: \Xset \to [1, \infty)$, constants $\beta \in (0,1)$, $b <
\infty$ and a measurable set $\Cset$.
The main result of this section is the following generalisation of
\cite[Theorem 5.26]{Connor-2007}.  It shows how a Foster-Lyapunov drift
condition may be established for a subsampled chain from knowledge of the rate
of convergence of the signed measures $\{P^n(x, \cdot) - P^n(x',\cdot), n\geq
0\}$.

\begin{theo}
\label{theo:Ergo2Subsample}
Assume that there exists a non-decreasing function $r : \nset \to (0, \infty)$
with $\lim_{k\to\infty} r(k) = \infty$, some measurable functions $W,V : \Xset
\to [1, \infty)$ and a constant $C< \infty$ such that
  \begin{align}
    & \forall (x,x') \in \Xset \times \Xset \eqsp, \qquad \qquad r(k) \ \ \|
    P^k(x, \cdot) - P^k(x', \cdot) \|_W \leq C \left( V(x) + V(x') \right)
    \eqsp,
    \label{theo:Ergo2Subsample:Cond1}
    \\
    & \exists x_0 \in \Xset \eqsp, \qquad \qquad \qquad \quad \sup_{k \geq 0} P^k W(x_0) <
    \infty \eqsp.
    \label{theo:Ergo2Subsample:Cond2}
  \end{align}
  Let $\beta \in (0,1)$ and $n : \Xset \to \nset$ satisfy $n(x) \geq r^{-1}
  \left(\frac{C}{\beta} \ \frac{V(x)}{W(x)} \right)$, where $r^{-1}(t) \eqdef
  \inf \{x \in \nset, r(x) \geq t\}$ denotes the generalized inverse of $r$.
  Then there exists $b< \infty$ such that $P^{n(x)} W (x) \leq \beta W(x) + b
  $.  In addition, for any $\beta < \beta' < 1$ with $\Cset \eqdef \{x \in
  \Xset, W(x) \leq b(\beta'-\beta)^{-1}\}$,
\begin{equation}
  \label{eq:resulting-drift}
   P^{n(x)} W \leq \beta' W +b\un_\Cset \,.
\end{equation}
\end{theo}
\begin{proof}
  From (\ref{theo:Ergo2Subsample:Cond1}) and (\ref{theo:Ergo2Subsample:Cond2}),
  we have for any $x \in \Xset$, $k \in \nset$,
\[
P^k W(x) \leq \frac{C}{r(k)} V(x) + P^k W(x_0) + \frac{C}{r(k)} V(x_0) \leq
\frac{C}{r(k)} V(x) + b \eqsp,
\]
where $b \eqdef \sup_{k \geq 0} P^k W(x_0) + C \frac{V(x_0)}{r(0)}$. By
definition of $x \mapsto n(x)$, $ C V(x) / r(n(x)) \leq \beta W(x)$.  This
yields $ P^{n(x)} W \leq \beta W + b = \beta' W + \left(\beta - \beta' \right)
W + b$, and $\left(\beta - \beta' \right) W + b \leq 0$ on $\Cset^c$.  Since
$\lim_{k\to\infty} r(k) = \infty$, the set $\{x \in \nset,\, r(x) \geq t\}$ is
non-empty whatever $t \geq 0$.
\end{proof}

\subsection{Uniformly  ergodic chains}
When assumption~(\ref{theo:Ergo2Subsample:Cond1}) holds for some \emph{bounded}
function $V$, we have $\lim_n \sup_{(x,x') \in \Xset \times \Xset} \| P^n(x,
\cdot) - P^n(x', \cdot) \|_{\tv} =0$.  Then classical results on the Dobrushin
coefficient imply that $P$ admits a unique invariant distribution $\pi$ and
$\sup_{x \in \Xset} \|P^n(x, \cdot) - \pi\|_{\tv} \leq \rho^n$ for some $\rho
\in (0,1)$.  Hence the drift condition (\ref{eq:FL-drift}) holds for some
bounded function $\tilde{V}$ (\cite[Theorem 16.0.1]{meyn:tweedie:1993}), and
this implies (\ref{eq:resulting-drift}) with $n(x)=c$ for any $c \in \nset$. We
are able to retrieve this result too from our result.  Applying
Theorem~\ref{theo:Ergo2Subsample} with $W=\omega$ and $n(x) = n_\star$ such
that $n_\star \geq r^{-1}\left( C \ \beta^{-1}\ \sup_\Xset V \right)$ we have
$P^{n_\star} W \leq \beta W +b$. By classical computations (see e.g. the proof
of \cite[Theorem 16.1.4]{meyn:tweedie:1993}), this yields $ P \tilde W \leq
\beta^{1/n_\star} \ \tilde W + b \beta^{1/n_\star-1} n_\star^{-1}$ with $\omega
\leq \tilde W \leq \omega \beta^{-1}$; hence (\ref{eq:resulting-drift}) holds
with $n(x) = c$ for some (and thus any) constant $c$.

In the sequel, we do not impose boundedness on $V$, thus allowing chains which
are not necessarily uniformly ergodic.

\subsection{Sufficient conditions for Assumptions \eqref{theo:Ergo2Subsample:Cond1}-\eqref{theo:Ergo2Subsample:Cond2}}\label{subsubsec:sufficient-conditions}
When $P$ is phi-irreducible and aperiodic,
assumption~(\ref{theo:Ergo2Subsample:Cond2}) is implied by any one of the
following equivalent conditions (see \cite[Theorem 14.0.1]{meyn:tweedie:1993}):
\begin{enumerate}[(i)]
\item $P$ possesses a unique invariant probability $\pi$ and $\pi(W) <
  \infty$;
\item there exists a small set $\Cset$ such that $\sup_{x \in \Cset} \PE_x
  \left[ \sum_{k=0}^{\tau_\Cset -1} W(\Phi_k) \right] < \infty$;
\item there exist a function $U : \Xset \to (0, \infty]$ finite at some
  $x_\star \in \Xset$, a constant $b < \infty$ and a small set $\Cset$ such
  that $PU \leq U - W + b \un_\Cset$.
\end{enumerate}
The main difficulty is to prove (\ref{theo:Ergo2Subsample:Cond1});
Proposition~\ref{prop:TwoMoments} provides sufficient conditions.

\begin{prop}
\label{prop:TwoMoments}
Let $P$ be a phi-irreducible and aperiodic transition kernel. Assume that there
exist a small set $\Cset$, measurable functions $W,V : \Xset \to [1, \infty)$
and a constant $b<\infty$ such that $\sup_\Cset V < \infty$, and
\begin{equation}
  \label{eq:DoubleControl}
\left\{
  \begin{array}{l}
PV(x)  \leq V(x) - W(x) + b \un_\Cset(x) \eqsp,  \\
PW(x)  \leq W(x)  + b \un_\Cset(x) \eqsp. 
  \end{array} \right.\end{equation}
Then \textit{(i)} there exists $x_0 \in \Xset$ such that $\sup_{n \geq 0} P^n
W(x_0) < \infty$; and \textit{(ii)} there exists a constant $C < \infty$ such
that for any $(x,x') \in \Xset^2$,
\[
n \ \|P^n(x, \cdot) - P^n(x', \cdot) \|_{W} \leq C \{V(x) + V(x')\}\,.
\]
\end{prop}
When the first inequality in (\ref{eq:DoubleControl}) holds, then it also holds
by replacing $W$ with $\tilde W \eqdef 1$, which trivially satisfies the second
inequality in (\ref{eq:DoubleControl}). In this setting
Proposition~\ref{prop:TwoMoments} is \cite[Theorem 13.4.4]{meyn:tweedie:1993}.
Nevertheless, the interest of satisfying the second inequality with an
unbounded function $W$ is that this allows for control of $ \|P^n(x, \cdot) -
P^n(x', \cdot) \|_{W}$ with a stronger norm than total variation.
Proposition~\ref{prop:TwoMoments} also provides a rate of convergence for $
\|P^n(x, \cdot) - P^n(x', \cdot) \|_{W}$ to zero; this rate is stronger than
that which could be deduced from the control of $\|P^n(x, \cdot) - \pi \|_{W}$ under similar assumptions~\cite[Chapter 14]{meyn:tweedie:1993}.

The following two results show that (\ref{eq:DoubleControl}) holds when we have
a subgeometric-type drift inequality, or when we are able to control modulated
moments of the return time to a small set $\Cset$.
A proof of Proposition~\ref{lemma:Drift2Ergod} can be found in \cite[Lemma 5.9]{Connor-2007}.

\begin{prop}
\label{lemma:Drift2Ergod}
Assume that there exist a set $\Cset$, a constant $b<\infty$, a measurable
function $V : \Xset \to [1,\infty)$ and a continuously differentiable
increasing concave function $\phi : [1,\infty) \to (0, \infty)$, such that
\[
PV \leq V - \phi \circ V + b \un_\Cset \eqsp, \qquad \qquad \sup_\Cset V <
\infty \eqsp, \qquad \qquad \inf_{[1,\infty)} \phi >0 \eqsp.
\]
Then (\ref{eq:DoubleControl}) holds with $W \propto \phi \circ V$.
\end{prop}

\begin{prop}
 \label{lemma:Moment2Ergod}
 Assume that there exist a set $\Cset$, a non decreasing rate function $r :
 \nset \to (0, \infty)$ such that $r(0) \geq 1$ and $\sup_{x \in \Cset} \PE_x
 \left[ \sum_{k=1}^{\tau_\Cset} r(k) \right]< \infty$.  Then
 (\ref{eq:DoubleControl}) holds with
\[
V(x) = \PE_x \left[ \sum_{k=0}^{\sigma_\Cset} r(k) \right] \eqsp, \qquad W(x) =
\PE_x \left[ r(\sigma_\Cset) \right]\eqsp.
\]
\end{prop}

\subsection{Examples}
\label{subsubsec:rate-scale-examples} A phi-irreducible
aperiodic chain satisfying the conditions of
Proposition~\ref{lemma:Drift2Ergod} for a small set $\Cset$, is ergodic at a
subgeometric rate~\cite{douc:fort:moulines:soulier:2004}. 
In that case conditions (\ref{theo:Ergo2Subsample:Cond1}) and
(\ref{theo:Ergo2Subsample:Cond2}) hold with $r(k) = k$ and $W \sim \phi \circ
V$, for which $r^{-1}\left( C \beta^{-1} V/W \right) \sim C \beta^{-1} V / \phi
\circ V$. This yields the following examples of subsampling rate $n$ and the
scale function $W$; hereafter, $c'>0$.

{\em Logarithmically ergodic chains.}  Assume that $\phi(t) \sim c \ [1+\ln
t]^\alpha$ for some $\alpha >0$ and $c >0$. Then (\ref{eq:resulting-drift})
holds with $n(x) \geq c' \ \frac{V(x)}{[1+\ln V(x)]^\alpha}$ and $W \eqdef [1 +
\ln V]^\alpha$.

{\em Polynomially ergodic chains.} Assume that $\phi(t) \sim c t^{1-\alpha}$
for some $\alpha \in (0,1)$ and $c >0$. Then (\ref{eq:resulting-drift}) holds
with $n(x) \geq \ c' V^{\alpha}(x)$ and $W \eqdef V^{1-\alpha}$.

{\em Subgeometrically ergodic chains.} Assume that $\phi(t) \sim c t [\ln
t]^{-\alpha}$ for some $\alpha>0$ and $c >0$. Then (\ref{eq:resulting-drift})
holds with $n(x) \geq \ c' [\ln V(x)]^{\alpha}$ and $W \eqdef V [\ln
V]^{-\alpha}$.

The results above are coherent with the geometric case: on one hand, when a
transition kernel satisfies the drift inequality $PV \leq \beta V + b
\un_\Cset$, then it also satisfies (\ref{eq:FL-drift}) with the same drift
function $V$, and $n(x) =c$ for any constant $c \in \nset_\star$; on the other
hand, when $\alpha \to 0$, the polynomial and subgeometric drift conditions
`tend' to the geometric drift condition (in the sense that $\phi(t) \to t$).
From the above discussion, when $\alpha \to 0$, $[c' V^{\alpha}, V^{1-\alpha}]
\to [c', V]$ and $[c' \, (\ln V)^{\alpha}, V (\ln V)^{-\alpha}] \to [c', V]$,
thus showing coherence in the results.

\section{State-dependent drift criteria for regularity }
\label{sec:Moment}
We now discuss how a state-dependent Foster-Lyapunov drift condition is related
to the existence of a moment of the return time to measurable sets.  Such
controls are related to the {\em regularity} of the chain (\cite[chapter
14]{meyn:tweedie:1993},
\cite{tuominen:tweedie:1994})  
which, under general conditions, is known to imply limit theorems such as
strong laws of large numbers, mean ergodic theorems, functional central limit
theorems and laws of iterated logarithm (see \cite[chapters 14 to
17]{meyn:tweedie:1993}). We provide conditions for the control of
(subgeometric) moments of the return time to a small set, expressed in terms of
a family of nested drift conditions (Proposition~\ref{Drift:sequence}) or in
terms of a single drift condition (Theorem~\ref{theo:subsampled2moments}).

\subsection{Family of nested drift conditions}
Proposition~\ref{Drift:sequence} extends the conditions provided by Tuominen
and Tweedie~\cite{tuominen:tweedie:1994}, expressed in terms of the one-step transition kernel, to the case of the state-dependent
transition kernels. 

\begin{prop}
\label{Drift:sequence}
Let $f : \Xset \to [1,\infty)$ and $n : \Xset \to \nset$ be measurable
functions and $\{r(k), k\geq 0\}$ be a non-negative sequence.  Assume that
there exist measurable functions $\{V_k , k\geq 0\}$ and $\{S_k, k\geq 0\}$,
$V_k,S_k : \Xset \to [1, \infty)$, and a measurable set $\Cset$ such that for
any $k \geq 0$, $x \in \Xset$,
\[
\PE_x \left[ V_{k + n(\Phi_0)}(\Phi_{n(\Phi_0)})\right] \leq V_k(x) - \PE_x
\left[\sum_{j=0}^{n(\Phi_0)-1} r(k+j) f(\Phi_j)\right] + S_k(x) \un_\Cset(x)
\eqsp.
\]
Then for any $x \in \Xset$, $\PE_x \left[ \sum_{k=0}^{\tau_\Cset -1} r(k) \ 
  f(\Phi_k) \right] \leq V_0(x) + S_0(x) \ \un_\Cset(x)$.
\end{prop}
However, this criterion is more of theoretical than practical interest since it is quite difficult to check. We now propose a criterion based on a single
drift condition.

\subsection{Single drift condition}
\label{ssec:singleDrift}
We consider the case when
\begin{equation}
  \label{eq:OurDrift-bis}
  \PE_x\left[ W(\Phi_{n(\Phi_0)}) \right] \leq \beta W(x) + b \un_\Cset(x) \eqsp,
\end{equation}
for some $\beta \in (0,1)$ and measurable positive functions $n,W \geq 1$.  The
case $n(x)=c$ on $\Xset$ corresponds to the usual Foster-Lyapunov drift
condition (see e.g. \cite[Chapter 15]{meyn:tweedie:1993} and references
therein).  The drift condition extends earlier work by Meyn and Tweedie
\cite[Theorem 2.1]{meyn:tweedie:1994} who address the cases when the drift
condition is of the form
\begin{align}
  & \PE_x\left[ W(\Phi_{n(\Phi_0)}) \right] \leq W(x) + b \un_\Cset(x) \eqsp, \nonumber \\
  & \PE_x\left[ W(\Phi_{n(\Phi_0)}) \right] \leq W(x) -n(x) + b \un_\Cset(x)\eqsp, \nonumber \\
  & \PE_x\left[ W(\Phi_{n(\Phi_0)}) \right] \leq \beta^{n(x)} W(x) + b
  \un_\Cset(x)\eqsp, \qquad \beta\in(0,1), \label{eq:statedep_geom}
\end{align}
without assuming any relations
between $n$ and $W$. In \cite{meyn:tweedie:1994}, it is established that for a
phi-irreducible and aperiodic kernel $P$, these drift inequalities imply
respectively Harris-recurrence, positive Harris-recurrence and geometric ergodicity
provided the set $\Cset$ is small and $W$ is bounded on $\Cset$. Our drift
inequality (\ref{eq:OurDrift-bis}) differs from (\ref{eq:statedep_geom}) in the
rate of contraction $\beta$ which does not depend on the subsampling rate
$n(x)$. 

\begin{theo}
  \label{theo:subsampled2moments}
  Assume that there exist measurable functions $W : \Xset \to [1,\infty)$ and
  $n : \Xset \to \nset_\star$, constants $\beta \in (0,1)$ and $b<\infty$, and
  a measurable set $\Cset$ such that (\ref{eq:OurDrift-bis}) holds. If there
  exists a strictly increasing function $R: (0,\infty)\rightarrow(0,\infty)$
  satisfying one of the following conditions
  \begin{enumerate}[(i)]
  \item \label{case1} $t \mapsto R(t)/t$ is non-increasing and $R \circ n \leq
    W$,
  \item \label{case2} $R$ is a convex continuously differentiable function such
    that $R'$ is log-concave and $R^{-1}(W) - R^{-1}(\beta W) \geq n$,
  \end{enumerate}
  then there exists a constant $C$ such that $ \PE_x [ R(\tau_\Cset) ] \leq C
  \{W(x) + b\un_\Cset (x) \}.$

\end{theo}
The conclusion of Theorem~\ref{theo:subsampled2moments} is unchanged if $R$ is
modified on some bounded interval $[0,t]$: hence it is sufficient to define $R$
such that the above conditions on $(R,n,W)$ hold for any $x$ such that $n(x)$
-- or equivalently $W(x)$ -- is large enough.  We provide at the end of this
section examples of pairs $(n,W)$ and the associated rate $R$.

A sufficient condition for the existence of $R$ satisfying (\ref{case1}) is
that, outside some level set of $n$, there exists a strictly increasing concave
function $\xi$ such that $\xi \circ n =W$. Then we can set $R= \xi$.  Since
concave functions are sub-linear, the case (\ref{case1}) addresses the case
when $n \gg W$ (outside some bounded set). A sufficient condition for the
existence of a function $R$ satisfying $R^{-1}(W) -R^{-1}(\beta W) \geq n$ is
that $(1-\beta) \; W(x) \; [R^{-1}]'(W(x)) \geq n(x)$; when there exists $\xi$
such that $n(x) = \xi \circ W(x)$ and $t \mapsto \xi(t)/t$ is non-increasing,
we can choose $R^{-1}(t) \sim \int_1^t u^{-1} \xi(u) du$. Hence case
(\ref{case2}) addresses the case when $n/W$ is decreasing (outside some bounded
set).

Existence of an invariant probability distribution is related to $\PE_x[
\tau_\Cset]$, the first moment of the return time to a small set $\Cset$
(\cite[Theorem 10.0.1]{meyn:tweedie:1993}).
Theorem~\ref{theo:subsampled2moments}(\ref{case1}) shows that the control of
this moment can be deduced from a condition of the form (\ref{eq:OurDrift-bis})
provided $n(x) \leq W(x)$ (choose $R(t) = t$).

Given $(n,W)$ and a drift inequality of the form (\ref{eq:OurDrift-bis}),
Theorem~\ref{theo:subsampled2moments} provides a moment of the return time to
$\Cset$ which depends upon the initial value $x$ at most as $W(x)$ (outside
$\Cset$). From the drift inequality (\ref{eq:OurDrift-bis}), we are able to
deduce a family of similar drift conditions with $n$ unchanged: for example,
Jensen's inequality implies that for any $\alpha \in (0,1)$,
\[
\PE_x \left[ W^\alpha(\Phi_{n(\Phi_0)})\right] \leq \beta^\alpha W^\alpha(x) + b^\alpha \un_\Cset(x) \eqsp.
\]
Application of Theorem~\ref{theo:subsampled2moments} with this new pair $(n,
W^\alpha)$ , will allow the control of a moment of $\tau_\Cset$ which depends
on $x$ at most as $W^\alpha(x)$.

\begin{coro}[of Theorem~\ref{theo:subsampled2moments}]
\label{coro:Pi-integ}
Assume in addition that $P$ is phi-irreducible and aperiodic, $\Cset$ is small
with $\sup_\Cset W < \infty$, and $R$ is a subgeometric rate.
  \begin{enumerate}[(i)]
  \item For any accessible set $\Dset$, there exists $C<\infty$ such that
    $\PE_x\left[ R(\tau_\Dset) \right] \leq C \ W(x)$;
  \item \label{cor:cond2} If $P$ admits a unique invariant probability measure
    $\pi$ such that $\pi(W)< \infty$, there exists an accessible small set
    $\Dset$ such that $\sup_{x \in \Dset} \PE_x\left[ \sum_{k=0}^{\tau_\Dset
        -1} R(k) \right] < \infty$.
  \end{enumerate}
\end{coro}
Examples of moments that can be obtained from
Theorem~\ref{theo:subsampled2moments} (possibly combined with
Corollary~\ref{coro:Pi-integ}) are the geometric, subgeometric, polynomial and
logarithmic rates.

{\em Geometric rates.} If $n(x)=1$: setting $R^{-1}(t) = \ln(t)/\ln(\kappa)$
with $1 \leq \kappa\leq\beta^{-1}$ it is easily verified that condition
\eqref{case2} of the theorem is satisfied. We therefore deduce that $\PE_x
[R(\tau_\Cset)] = \PE_x [ \kappa^{\tau_\Cset} ] \leq C \{W(x) + b\un_\Cset \}$.
In particular, $\PE_x [ \beta^{-\tau_\Cset} ]<\infty$, in agreement with the
well-known equivalence between the geometric drift condition and the
exponential moment of $\tau_\Cset$ mentioned in Section~\ref{sec:intro}.

{\em Subgeometric rates.}  If $n(x) \propto [\ln V(x)]^{\alpha}$ for some
$\alpha >0$ and $W \propto V[\ln V]^{-\alpha}$: then $n(x) \sim \xi \circ
W(x)$ with $\xi(t) \sim [\ln t]^\alpha$. For some convenient $c$, the function
$R(t) \sim \exp(c t^{1/(1+\alpha)})$ satisfies the condition $[R^{-1}]'(t) \sim
\xi(t)/[(1-\beta) t]$ and also condition~\eqref{case2} of
Theorem~\ref{theo:subsampled2moments}.

{\em Polynomial rates.}  If $n(x) \propto V^\alpha(x)$, for some $\alpha \in
(0,1]$ and $W \propto V^{1-\alpha}$: when $\alpha \leq 1/2$ (respectively
$\alpha \geq 1/2$) condition \eqref{case2} (resp. condition \eqref{case1}) of
Theorem~\ref{theo:subsampled2moments} is satisfied with $R(t) \sim
t^{(1-\alpha)/\alpha}$. We thus have $\PE_x \left[ \tau_\Cset^{1/\alpha
    -1}\right] \leq C \ V^{1-\alpha}(x) $.

{\em Logarithmic rates.}  If $n(x) \propto V [\ln V(x)]^{-\alpha}$ for some
$\alpha >0$ and $W \propto [\ln V]^{\alpha}$. Then $n\gg W$ and condition
\eqref{case1} is verified with $R(t) \sim [\ln t]^\alpha$.  Hence $\PE_x \left[
  (\ln \tau_\Cset)^\alpha\right] \leq C \ [1+\ln V]^{\alpha}(x) $.

As an application of Corollary~\ref{coro:Pi-integ} and of the discussion in
Section~\ref{subsubsec:rate-scale-examples}, we can deduce moments of the
return-time to $\Cset$ from a single drift condition of the form $PV \leq V
-\phi \circ V + b \un_\Cset$, $\phi$ concave. For example, in the case $\phi(t)
\sim t^{1-\alpha}$ for some $\alpha \in (0,1)$, we have
$\PE_x[\tau_\Cset^{1/\alpha}] \leq C V(x)$. This is in total agreement with that
which has been established in the literature using other approaches
\cite{jarner:roberts:2002,fort:moulines:2003,douc:fort:moulines:soulier:2004}.
This agreement illustrates the fact that the sufficient conditions provided in
Sections~\ref{sec:Tame} and \ref{sec:Moment} are quite minimal.

\section{Application to tame chains}
\label{sec:applications}

\subsection{Tame chains}\label{subsec:tame}
The class of \emph{tame} Markov chains was introduced by Connor and
Kendall~\cite{Connor.Kendall-2007}, who showed that a perfect simulation
algorithm exists for such chains.

\begin{defi}\label{defi:tame}
The chain $\Phi$ is \emph{tame} if there exists a scale function $W:\Xset\to[1,\infty)$, a small set $\Cset$, and constants $\beta\in(0,1)$, $b<\infty$ such that the following two conditions hold:
\begin{enumerate}[(i)]
\item \label{def1}
there exist $\delta\in(0,1)$ and a deterministic function $n:\Xset\to\nset$ satisfying 
\begin{align}
  n(x) & \leq  W^\delta(x)  \label{eq:nBound} \\
\text{such that} \qquad  \PE_x\left[ W(\Phi_{n(x)}) \right] &\leq \beta W(x) + b \un_\Cset(x) \eqsp;   \label{eq:OurDrift}
\end{align}
\item \label{def2} the constant $\delta$ in \eqref{eq:nBound} satisfies $\ln \beta < \delta^{-1}\ln (1-\delta)$.
\end{enumerate}
\end{defi}

In other words, $\Phi$ is tame if for all $x\in\Xset$ there exists a deterministic time $n(x)$ such
that the chain subsampled at this time exhibits a geometric drift condition
\eqref{eq:OurDrift}. Furthermore, $n(x)$ should be sufficiently small compared
to the scale function $W$ \eqref{eq:nBound}. Part
(\ref{def2}) of Definition~\ref{defi:tame} is a technical condition required
for construction of the simulation algorithm.

Clearly the class of tame chains includes all geometrically ergodic chains. In addition, it is shown in \cite{Connor.Kendall-2007} that $\Phi$ is tame if it satisfies a polynomial drift condition of the form 
\begin{equation}
  \label{eq:polyDrift}
  PV \leq V-cV^{(1-\alpha)}+b \un_C \,,
\end{equation} 
where $0<\alpha<1/4$. 
However, this condition is not necessary: in \cite{Connor.Kendall-2007} there is an example of a random walk satisfying
\eqref{eq:polyDrift} with $\alpha=1/2$, which is explicitly shown to be tame. The results of Section~\ref{subsubsec:sufficient-conditions} now enable us to generalise this sufficient condition.

\begin{prop}
  \label{prop:tameness}
  Suppose that $\Phi$ satisfies the assumptions of
  Proposition~\ref{lemma:Drift2Ergod} with $\phi(t)\sim c t^{1-\alpha}$
  where $\alpha \in (0,1/2)$. Then $\Phi$ is tame.
\end{prop}

\begin{proof}
  Choose $\delta\in(0,1)$ such that $\delta>\alpha/(1-\alpha)$, and then
  $\beta\in(0,1)$ such that $\ln\beta < \delta^{-1}\ln(1-\delta)$. As noted in
  Section~\ref{subsubsec:rate-scale-examples}, the results of
  Proposition~\ref{lemma:Drift2Ergod} and Theorem~\ref{theo:Ergo2Subsample}
  show that if $\phi(t)\sim ct^{1-\alpha}$ for $\alpha<1/2$, then $ \PE_x\left[
    W(\Phi_{n(x)}) \right] \leq \beta W(x) + b \un_\Cset(x)$ with
  $\displaystyle{ n(x) = c_\beta \ V^{\alpha}(x) \eqsp, \eqsp W = V^{1-\alpha}
    \eqsp,}$ and where $c_\beta \propto \beta^{-1}$. The choice of $\delta$ and $\beta$ ensures that $\Phi$ satisfies all parts of Definition~\ref{defi:tame}, as required.
\end{proof}

Note that any chain with subgeometric drift $\phi(t) \sim c t [\ln
t]^{-\alpha}$ is tame. However, the logarithmically ergodic chains identified
in Section~\ref{subsubsec:rate-scale-examples} do not satisfy $\phi(t)>t^\alpha$ for any value of $\alpha\in(0,1)$, and
so are not covered by Proposition~\ref{prop:tameness}.


\subsection{Dominating process for tame chains }
\label{sec:applications-2}
In this section we describe a non-trivial example of a Markov chain $D$ for
which there is no obvious one-step drift, but for which it is simple to
establish a sub-sampled drift condition. The chain $D$ finds application in the
perfect simulation algorithm of
\cite{Connor.Kendall-2007}, as will be explained below.

Let $\beta \in (0, e^{-1})$, $\kappa>0$ and $n^\star$ be a function from $[1,
\infty) \to \nset$. To begin our construction of $D$, let $U$ be the system
workload of a $D/M/1$ queue, sampled just before arrivals, with arrivals every
$\ln \beta^{-1}$ units of time, and service times being independent and of unit
rate Exponential distribution. This satisfies
\[ U_{n+1} = \max\left\{U_n + E_{n+1} - \ln \beta^{-1},\,0\right\} \,, \]
where $\{E_n\}_{n\geq 1}\iid \textup{Exp}(1)$.

Define $Y = \kappa\exp(U)$, for some $\kappa>0$. The set
$[\kappa,\kappa/\beta]$ is a small set for $Y$, and 
\begin{equation}
  \label{eq:Y-density}
  \PP[Y_1>v \;|\; Y_0=u] = \frac{\beta u}{v}\,, \quad\text{for $v\geq\max\{\beta u, \kappa \}$.}
\end{equation}

Finally, let $D$ be the two-dimensional process $D=(Z,M)$ on $\Xset \eqdef \{(z,m): z \in [1, \infty), m \in \{1, \dots, n^\star(z)\} \}$ with transitions controlled by:
    \begin{eqnarray*}\label{eqn:(Y,N)-trans-probs}
        && \PP[Z_{k+1}=Z_k,\, M_{k+1}=M_k-1 \;|\; Z_k,\, M_k] = 1\;,
        \quad\text{if $M_k\geq 2$\,;} \\
        && \PP[Z_{k+1}\in E \;|\; Z_k=z,\, M_k=1] = \PP[Y_1\in E \;|\;
        Y_0=z]\,, \\
        && \hspace{5.5cm} \text{for all measurable $E\subseteq [1,\infty)$;} \\
        && \PP[M_{k+1}=n^*(Z_{k+1}) \;|\; Z_k,\, Z_{k+1}\,, M_k=1] = 1\,.
    \end{eqnarray*}
   Thus the first component of $D$ is
      simply a slowed down version of $Y$, and the second component is a
      forward recurrence time chain, counting down the time until the first
      component jumps (determined by the function $n^*$).
      
\begin{prop}
      \label{prop:tame-domination}
      Let $\beta \in (0,e^{-1})$, $\kappa>0$ and $n^*$ be a measurable function
      $n^\star: [1, \infty) \to\nset$. Set $\Cset \eqdef \{(z,m): z \in
      [\kappa, \kappa/\beta], m \in \{1, \cdots, n^\star(z) \} \}$ and denote
      by $\tau_\Cset$ the return time to $\Cset$ for the Markov chain $D$. Let
      $\alpha_\beta$ be the unique solution in $(0,1)$ of the equation $\ln
      \beta = \ln(1-\alpha)/\alpha$.
      \begin{enumerate}[(i)]
      \item When $n^\star(z) \sim z^{\gamma}$ for some $\gamma \geq 0$ : for
        any $ \alpha \in (0, \alpha_\beta)$ and
        \begin{itemize}
        \item any $\eta \in (\gamma/\alpha,1]$ when $\gamma \in [0, \alpha_\beta)$
        \item any $\eta > \gamma/\alpha$ when $\gamma \geq \alpha_\beta$,
        \end{itemize}
        there exists a constant $C$ such that $\PE_{(z,m)}[\tau_\Cset^{1/\eta}]
        \leq C \, z^{\alpha}$ for any $(z,m) \in \Xset$.
      \item When $n^\star(z) \sim [\ln z]^\gamma$ for some $\gamma > 0$ : for
        any $ \alpha \in (0, \alpha_\beta)$ and $\eta>0$ satisfying
        $\eta < \{ (1+\gamma)\alpha^{-1} \ln((1-\alpha)/\beta^\alpha)\}^{1/(1+\gamma)}$, there
        exists a constant $C$ such that $\PE_{(z,m)}[\exp (\eta \alpha
        \tau_\Cset^{1/(1+\gamma)})] \leq C \, z^{\alpha}$ for any $(z,m) \in
        \Xset$.
      \item When $n^\star(z) \sim 1$, for any $ \alpha \in (0, \alpha_\beta)$
        there exists a constant $C$ such that for any $(z,m) \in \Xset$,
        $\PE_{(z,m)}[ \left\{(1-\alpha) \beta^{-\alpha} \right\}^{\tau_\Cset}] \leq C \,
      z^{\alpha}$
      \end{enumerate}
  \end{prop}

  


\begin{proof}
  We first of all establish a drift condition of the form
  \eqref{eq:OurDrift-bis} for the chain $D$
  .  Let $V(z,m) = z^\alpha$,
  with $\alpha \in (0,\alpha_\beta)$. Then 
  $\PE[V(Z_m,M_m) \;|\; Z_0=z, M_0=m] = \PE[\,Y_1^\alpha
  \;|\; Y_0=z]$. When $z\notin [\kappa,\kappa/\beta]$,
\[
\PE[V(Z_m,M_m) \;|\; Z_0=z, M_0=m] = \int_{\beta z}^\infty y^\alpha \frac{\beta
  z}{y^2}\, dy = \frac{\beta^\alpha}{1-\alpha}z^\alpha \,.
\]
Since $\alpha < \alpha_\beta$, $\alpha \ln \beta <
  \ln(1-\alpha)$ and the chain $D$ satisfies the drift condition
\begin{equation}
  \label{eq:D-drift}
  P^m V(z,m) \leq \beta' V(z,m)\,, \quad\text{with $\beta' = \frac{\beta^\alpha}{1-\alpha}<1$,}
\end{equation}
whenever $z\notin[\kappa,\kappa/\beta]$. If $z\in[\kappa,\kappa/\beta]$ however, then
\begin{align*}
  \PE[V(Z_m,M_m) \;|\; Z_0=z, M_0=m]  & = \int_\kappa^\infty y^\alpha \frac{\beta z}{y^2}\, dy +\left( 1-\frac{\beta z}{\kappa}\right)\kappa^\alpha \\
  &\leq \beta' z^\alpha +\left(\frac{1-\beta^\alpha}{1-\alpha}\right)
  \kappa^\alpha \,.
\end{align*}
It follows that $D$ satisfies $P^m V(z,m) \leq \beta'V(z,m)+b'\un_{\Cset}$.

We may now apply
Theorem~\ref{theo:subsampled2moments} to establish moments of the return time
of $D$ to the  set $\Cset$. 
For example, suppose that $n^*(z)\sim z^\gamma$, for
some $\gamma\in [0,\alpha_\beta)$ and set $R(z) \eqdef z^{1/\eta}$ for some
$\eta \in (\gamma /\alpha,1]$. It follows that $R$ satisfies the conditions of
Theorem~\ref{theo:subsampled2moments}(\ref{case2}), with
\begin{align*}
 R^{-1}(V(z,m)) - R^{-1}(\beta'V(z,m)) - m & \geq R^{-1}(z^\alpha) - R^{-1}(\beta'z^\alpha) - n^*(z) \\
& = \left(1 - \beta'^\eta\right) z^{\alpha\eta} - z^\gamma \geq 0 \,.
\end{align*}
(Here we have used the fact that $m\leq n^*(z)$, by definition of $\Xset$.)
Thus, 
\[ \PE_{(z,m)} [\,\tau_{\Cset}^{1/\eta}\, ] \leq C\{ z^\alpha + b'\un_{\Cset}(z,m)\} \,. \]

If $n^*(z)\sim z^\gamma$ with $\gamma \geq \alpha_\beta$ then the same argument shows that the function $R(z) = z^{1/\eta}$, with
$\eta>\gamma/\alpha$, satisfies
Theorem~\ref{theo:subsampled2moments}(\ref{case1}). Parts (ii) and (iii) follow similarly by taking $R(z)\sim \exp(\eta\alpha z^{1/(1+\gamma)})$ and $R(z)\sim
 \exp(\eta\alpha z)$ (for some $\eta>0$) respectively.
\end{proof}

In Proposition~\ref{prop:tame-domination}(i), $1/\eta \geq 1$ iff $\gamma \in
[0, \alpha_\beta)$. When $1/\eta \geq 1$ and $D$ is phi-irreducible, aperiodic
and $\Cset$ is small, this shows that $D$ possesses an invariant probability
distribution and is ergodic. The convergence to $\pi$ (in total variation norm)
occurs at the polynomial rate $n^{1/\eta-1}$ (see
\cite{tuominen:tweedie:1994}). When $1/\eta <1$, we cannot deduce from the
control of this moment the existence of $\pi$.

The chain $D$ is of interest for the following reason. Suppose that
   $\Phi$ is a tame chain satisfying $P^{n(x)}W(x) \leq
      \beta W(x) +b\un_{\Cset'}$ where $n(x) = n^\star
        \circ W(x)\leq W^\delta(x)$ for some $\delta\in(0,1)$.
Connor and Kendall~\cite{Connor.Kendall-2007} show that the chains $Z$ and $W(\Phi)$ can be
coupled so that $Z$ dominates $W(\Phi)$ at the times when $Z$ jumps. Thus the chain $D$ `pseudo-dominates' $W(\Phi)$, and this coupling can be
exploited to produce a perfect simulation algorithm for $\Phi$. Proposition~\ref{prop:tame-domination} allows us to calculate ergodic properties of $D$, and hence bound the expected run-time of the algorithm: this issue is not addressed in~\cite{Connor.Kendall-2007}. 

\section{Subgeometric ergodicity  of strong Markov processes}
\label{sec:appli-markov}
In this section  we provide sufficient conditions for ergodicity of a strong Markov process.  In
\cite{Meyn:Tweedie:1993c,Fort:Roberts:2005,Douc:Fort:Guillin:2008}, the
conditions are (mainly) expressed in terms of a drift inequality on the
generator of the process. Our key assumption A\ref{ex:A2} is in terms of the
time the process rescaled in time and space enters a ball of radius $\rho$,
$\rho \in (0,1)$.  Proposition~\ref{prop:limfl} finds application in, for example, queuing theory as discussed below.


Let $\{\Phi_{t}, t \in \rset_+\}$ be a strong Markov process taking values
in $\Xset \subseteq \rset^d$.  It is assumed that $(\Omega, \mathcal{A},
\mathcal{F}_t, \Phi_t, \PP_x)$ is a Borel-right process on the space $\Xset$
endowed with its Borel $\sigma$-field $\mathcal{B}(\Xset)$.  We assume that the
sub-level sets $\{x \in \Xset,\, |x| \leq \ell \}$ are compact subsets of
$\Xset$ ($|\cdot|$ is a norm on $\rset^d$).  \debutA
\item \label{ex:A2} $\lim_{|x| \to \infty} |x|^{-(p+1)} \PE_x \left[ |\Phi_{
      \lfloor t_0 |x|^{1+\tau} \rfloor }|^{p+1} \right] = 0$ for some $t_0>0$,
  $p \geq 0$ and $0 \leq \tau \leq p$.
\item \label{ex:A4} For any $t_\star>0$, there exists $C$ such that for any $x
  \in \Xset$, $\sup_{t \leq t_\star} \int P^t(x,dy) |y|^{p+1} \leq C \,
  |x|^{p+1} $.
\item \label{ex:A1} Every compact subset of $\Xset$ is small for the process and
  the skeleton $P$ is phi-irreducible.
\item \label{ex:A5} There exist $q \geq 0$ and $C$ such that for any $x$,
  $\PE_x \left[ \sum_{k=0}^{ \lfloor t_0 |x|^{1+\tau} \rfloor -1} |\Phi_{k
    }|^q\right] \leq C |x|^{p+1}$.  \finA Recall that a set $\Cset$ is said to
  be small (for the process) if there exist $t>0$ and a measure $\nu$ on
  $\mathcal{B}(\Xset)$ such that $P^t(x, \cdot) \geq \un_\Cset(x) \nu(\cdot)$.
  
  A\ref{ex:A2} is a condition on the process $\{\Phi_t, t \geq 0\}$ rescaled in
  time and space. Such a transformation is largely used in the queueing
  literature for the study of the stability of networks. This approach is
  refered to as the {\em fluid model} (see e.g.  \cite{robert:2003,meyn:2008} for
  a rigorous definition; see also \cite{dai:meyn:1994,dai:meyn:1995,meyn:2008}
  and references therein for applications to queueing). In these applications,
  A\ref{ex:A2} is proved by showing that the fluid model is stable (see e.g.
  \cite[Proposition 5.1]{dai:meyn:1995}). Condition A\ref{ex:A4} is a control
  of the $L^p$-moment of the system. A\ref{ex:A1} is related to the
  phi-irreducibility of the Markov process, a property which is necessary when
  ergodicity holds.  A\ref{ex:A5} is required to prove the existence of a
  steady-state value for the moments $\PE_{x}[|\Phi_t|^s]$ for $s >0$, when $t
  \to \infty$. Examples of Markov processes satisfying  A\ref{ex:A2}-\ref{ex:A5} are given in \cite{dai:meyn:1994,dai:meyn:1995}.

\begin{prop}
\label{prop:limfl}
Assume that A\ref{ex:A2}-\ref{ex:A1} hold. Then the Markov process possesses a
unique invariant probability $\pi$ and for any $x \in \Xset$, $\lim_{t \to
  \infty} (t+1)^{(p-\tau)/(1+\tau) } \ \|P^t(x,\cdot) - \pi(\cdot) \|_{\tv}
=0$.  If in addition A\ref{ex:A5} holds, then $ \int |y|^q \pi(dy)< \infty$ and
for any $x \in \Xset$ and any $0 \leq \kappa \leq 1$,
\[
\lim_{t \to + \infty} (t+1)^{\kappa (p-\tau)/(1+\tau) } \ \sup_{\{g: |g(x)|
  \leq 1 + |x|^{(1-\kappa)q} \}} \left| \PE_x \left[g(X_t) \right] - \pi(g)
\right| =0 \eqsp.
\]
\end{prop}
Proposition~\ref{prop:limfl} provides a polynomial rate of convergence and
convergence of power moments. More general rates of convergence can be obtained
by replacing in A\ref{ex:A2}-\ref{ex:A4} the power functions $|x|^{1+\tau}$,
$|x|^{p+1}$ by more general functions $W$; more general moments can be obtained
by replacing the power function $|x|^{q}$ in A\ref{ex:A5} by a general function
$f$. These extensions are easily obtained from the proof of
Proposition~\ref{prop:limfl}; details are omitted.

Proposition~\ref{prop:limfl} extends \cite[Theorem 3.1]{dai:1995} which
addresses the positive Harris-recurrence of the process. It also extends
\cite[Theorem 6.3]{dai:meyn:1995} by providing \textit{(i)} a continuum range
of rates of convergence (and thus a continuum range of rate functions) and
\textit{(ii)} an explicit norm of convergence.

\begin{proof}{\em of Proposition~\ref{prop:limfl}}.
  The reader unfamiliar with basic results on Markov processes may refer to
  \cite{meyn:tweedie:1993:art}.  For a measurable set $\Cset$ and a delay
  $\delta >0$, let $\tau_\Cset(\delta) \eqdef \inf \{t \geq \delta, \Phi_t \in
  \Cset \}$ denote the $\delta$-delayed hitting-time on $\Cset$; by convention,
  we write $\tau_\Cset$ for $\tau_\Cset(0)$.  Let $\beta \in (0,1)$ and set
  $W(x) \eqdef 1 + |x|^{p+1}$.  By A\ref{ex:A2}, there exists $\ell>0$ such
  that $t_0 \ell \in \nset$ and for any $x \notin \mathcal{C} \eqdef \{x,
  |x|^{1+\tau} \leq \ell \}$, $\PE_x \left[ \left| \Phi_{ \lfloor t_0
        |x|^{1+\tau} \rfloor}\right|^{p+1} \right] \leq 0.5 \, \beta \,
  |x|^{p+1}$.  We can assume without loss of generality that $\ell$ is large
  enough so that $\PE_x \left[ W(\Phi_{ \lfloor t_0 |x|^{1+\tau} \rfloor})
  \right] \leq \beta \, W(x)$ for $x \notin \Cset$.  Set $n(x) \eqdef \max(\ell
  t_0, \lfloor t_0 |x|^{1+\tau} \rfloor)$. By A\ref{ex:A4}, there exists
  $b<\infty$ s.t.
  \begin{equation}
    \label{eq:Drift:Queueing}
    \PE_x \left[ W(\Phi_{n(x)}) \right]  \leq \beta \, W(x) + b \un
  _\Cset(x) \eqsp.
  \end{equation}
  By Theorem~\ref{theo:subsampled2moments}, there exists $C<\infty$ such that
  $\PE_x \left[ \{\tau_{\star,\Cset}\}^{(p+1)/(\tau+1)} \right] \leq C W(x)$
  where $\tau_{\star,\Cset} \eqdef \inf \{n \geq 1, \Phi_{n} \in \Cset \}$ is
  the return time to $\Cset$ of the skeleton $P$.  Hence, there exists a delay
  $0<\delta \leq 1$ such that $\PE_x[\tau_\Cset(\delta) ] \leq C \, W(x)$;
  $\sup_\Cset W < \infty$ and $\Cset$ is small for the process (by
  A\ref{ex:A1}), so $\{\Phi_t, t\geq 0\}$ is positive Harris-recurrent and
  possesses a unique invariant probability measure $\pi$.  By~\cite[Proposition
  6.1]{meyn:tweedie:1993b} and \cite[Section 5.4.3]{meyn:tweedie:1993}, the
  skeleton $P$ is aperiodic and any compact set is small for the skeleton $P$.
  A\ref{ex:A1} and the above properties on the skeleton $P$ imply $\lim_{n \in
    \nset} (n+1)^{(p-\tau)/(1+\tau)} \ \| P^{n}(x,\cdot) - \pi(\cdot) \|_{\tv}
  =0 $ for any $x \in \Xset$ \cite[Theorem 2.1]{tuominen:tweedie:1994}, which
  in turn implies that $\lim_{t \in \rset^+} (t+1)^{(p-\tau)/(1+\tau)} \ \|
  P^{t }(x,\cdot) - \pi(\cdot) \|_{\tv} =0$.  Set $f(x) \eqdef 1 + |x|^q$.
  A\ref{ex:A5} and (\ref{eq:Drift:Queueing}) imply that
  \begin{multline*}
    \PE_x \left[ \sum_{k=0}^{\tau_{\star,\Cset}-1} f(\Phi_{k}) \right] \leq
    \PE_x \left[ \sum_{k=0}^{\bar{\tau}_{\Cset}-1} \PE_{ \bar \Phi_{k}} \left[
        \sum_{k=0}^{n(\bar \Phi_{k})-1} f(\Phi_k) \right] \right] \leq C\,
    \PE_x \left[ \sum_{k=0}^{\bar{\tau}_{\Cset}-1} W\left( \bar \Phi_{k}
      \right) \right] \leq C' \ W(x)
  \end{multline*}
  where $\bar{\tau}_{\Cset}$ is the return time to $\Cset$ of the discrete-time
  chain $\{\bar \Phi_n, n \geq 0\}$ with transition kernel $P^{n(x)}(x,\cdot)$.
  Hence $\pi(f)<\infty$ by \cite[Theorem 14.3.7]{meyn:tweedie:1993}.  As in
  \cite{douc:fort:moulines:soulier:2004}, we obtain $(f,r)$-modulated moments
  of $\tau_{\star,\Cset}$ by using Young's inequality (see e.g.
  \cite{krasnoselskii:rutickii:1961}). This yields ergodic properties for the
  skeleton $P$ and the desired limits for $\{P^t, t \geq 0\}$.
\end{proof}

\section{Appendix: Proofs}
\label{sec:Proof}

\subsection{Proof of Proposition~\ref{prop:TwoMoments}}
In the proof, $C$ is constant and its value may change upon each appearance. We
use the following properties \debutR
\item \label{R1} If $\Dset$ is petite for a phi-irreducible and aperiodic transition
  kernel, then it is also small (\cite[Theorem 5.5.7]{meyn:tweedie:1993}).
\item \label{R1bis} If a transition kernel is phi-irreducible and aperiodic,
  then any skeleton is phi-irreducible and aperiodic (\cite[Proposition
  5.4.5]{meyn:tweedie:1993}).
\item \label{R2} If $\Dset$ is $\nu$-small for a phi-irreducible and aperiodic transition
  kernel, then we can assume without loss of generality that $\nu$ is a maximal
  irreducibility measure (\cite[Proposition 5.5.5]{meyn:tweedie:1993}).
\item \label{R3} If there exist measurable functions $f,V : \Xset \to
  [1,\infty)$, a measurable set $\Cset$ and a constant $b$ such that $PV \leq
  V - f + b\un_\Cset$, then for any stopping-time $\tau$ (\cite[Proposition
  11.3.2]{meyn:tweedie:1993})
\[
\PE_x \left[ \tau \right] \leq \PE_x \left[ \sum_{k=0}^{\tau-1} f(\Phi_k)
\right] \leq V(x) + b \PE_x \left[ \sum_{k=0}^{\tau-1} \un_\Cset(\Phi_k)
\right] \eqsp.
\]
If in addition, there exist $m \geq 1$ and $c>0$ such that $c \un_\Cset(x) \leq
P^m(x,\Dset)$, then
\[
\PE_x \left[ \tau_\Dset \right] \leq \PE_x \left[ \sum_{k=0}^{\tau_\Dset-1}
  f(\Phi_k) \right] \leq V(x) + \frac{b}{c} \ \PE_x \left[
  \sum_{k=0}^{\tau_\Dset-1} \un_\Dset(\Phi_{k+m}) \right] \leq V(x) + \frac{b
  (m+1)}{c} \eqsp.
\]
 \finR

 The conditions (\ref{eq:DoubleControl}), $\sup_\Cset V< \infty$ and R\ref{R3}
 imply that $\sup_{x \in \Cset} \PE_x[\tau_\Cset] < \infty$.  Hence, $\Cset$ is
 an accessible set. By R\ref{R1} and R\ref{R2}, there exist $m \geq 1$ and a
 maximal irreducibility measure $\nu_m$ such that $P^m(x,\cdot) \geq
 \un_\Cset(x) \nu_m(\cdot)$.

 {\bf Step 1: From $P$ to the strongly aperiodic transition kernel $P^m$.}  By
 (\ref{eq:DoubleControl}), there exists a constant $C$ such that $W^{(m)}
 \eqdef \sum_{k=0}^{m-1} P^k W \leq C \ W$.  Write $n = k m +l $ for $ k \in
 \nset$ and $l \in \{0, \cdots, m-1\}$ so that
\begin{multline*}
  n \ \|P^n(x, \cdot) - P^n(x', \cdot) \|_W \leq n \ \|P^{km }(x, \cdot) -
  P^{km}(x', \cdot) \|_{W^{(m)}} \\
  \leq C \ (k+1) \ \|P^{km }(x, \cdot) - P^{km}(x', \cdot) \|_{W} \eqsp.
\end{multline*}
By R\ref{R1bis}, $P^m$ is phi-irreducible and strongly aperiodic, and satisfies
drift inequalities of the form (\ref{eq:DoubleControl}). Indeed,
\begin{multline*}
  P^m V(x) \leq V(x) - \sum_{k=0}^{m-1} P^k W(x) + b \sum_{k=0}^{m-1}
  P^k(x,\Cset) \leq V(x) - W(x) + b \sum_{k=0}^{m-1} P^k(x,\Cset) \eqsp;
\end{multline*}
this implies that there exist a small set $\Dset$ (for $P^m$) and a constant
$\bar b < \infty$ such that $P^m V(x) \leq V(x) - 0.5 \ W(x) + \bar b
\un_\Dset(x)$ (see e.g.\cite[proof of Lemma 14.2.8]{meyn:tweedie:1993}).  By
R\ref{R3}, for any accessible set $\mathcal{A}$ (for $P^m$), there exists $C$
such that
\begin{equation}
  \label{eq:MomentforPm}
  \PE_x[\tau_{\mathcal{A}}^{(m)}] \leq C V(x) \eqsp,
\end{equation}
where $\tau_{\mathcal{A}}^{(m)} \eqdef \inf \{n \geq 1, \Phi_{nm} \in
{\mathcal{A}} \}$.  Furthermore, we also have $P^m W(x) \leq W(x) + b
\sum_{k=0}^{m-1} P^k(x,\Cset)$.  This implies for any accessible (for $P^m$)
measurable set $\mathcal{A}$
\begin{equation}
  \label{eq:ControlforPm}
  (k+1) \ \PE_{x} \left[ W(\Phi_{km}) \un_{k \leq
    \tau_{\mathcal{A}}^{(m)} }\right] \leq C V(x) \eqsp;
\end{equation}
(the proof is postponed below).
Following the same approach as in \cite{tuominen:tweedie:1994} or \cite[Chapter
14]{meyn:tweedie:1993}, we use the {\em splitting technique} and associate to
the chain with transition kernel $P^m$ a split chain that possesses an atom.

{\bf Step 2: From $P^m$ to an atomic transition kernel $\cP$.}  A detailed
construction of the split chain and the connection between the original chain
and the split chain can be found in \cite[Chapter 5]{meyn:tweedie:1993}. We use
the same notation as in \cite{meyn:tweedie:1993}. The split chain $\{(\Phi_n,
d_n), n \geq 0\}$ is a chain taking values in $\Xset \times \{0,1\}$: its
transition kernel is denoted by $\cP$. 
Based on the connection between $P^m$ and $\cP$,
Proposition~\ref{prop:TwoMoments} holds provided there exists $C
< \infty$ with
\begin{equation}
  \label{eq:GoalAtomic}
  m \ (n+1) \ \sup_{\{f, \sup_\Xset |f| [ W]^{-1} \leq 1 \}} \ | \cP^n f(x) -
\cP^n f(y)| \leq C \ \{V(x) +V(y)\} \eqsp.
\end{equation}
We prove~(\ref{eq:GoalAtomic}). $\cP$ is phi-irreducible and aperiodic and
possesses an accessible atom $\alpha \eqdef \Cset \times \{1\}$.  Let
$\tau_\alpha$ be the return time to the atom $\alpha$.  From
(\ref{eq:MomentforPm}), we have $\cPE_{x^\star}\left[ \tau_\alpha\right] \leq C
V(x)$ (see e.g. \cite[Proposition 3.7]{tuominen:tweedie:1994} or \cite[Lemma
2.9]{nummelin:tuominen:1983}); and by (\ref{eq:ControlforPm})
\begin{equation}
  \label{eq:ControlforAtom}
  (k+1) \ 
\cPE_{x^\star} \left[ W(\Phi_{k}) \un_{k \leq \tau_\alpha }\right] \leq C V(x)
\eqsp.
\end{equation}
(The proof of (\ref{eq:ControlforAtom}) is postponed below.)\\
Set $a_x(n) \eqdef \cPP_{x^\star}( \tau_\alpha = n)$, $u(n) \eqdef \cPP_\alpha(
(\Phi_n, d_n) \in \alpha)$ and $t_f(n) \eqdef \cPE_{\alpha}\left[ f(\Phi_n)
  \un_{n \leq \tau_\alpha} \right]$, where $a \ast b(n) \eqdef \sum_{k=0}^{n}
a(k) b(n-k)$.  Then, by the first-entrance last-exit decomposition
\cite[Chapter 14]{meyn:tweedie:1993}, for any function $f$ such that $|f| \leq
W$,
\begin{multline*}
  (n+1) \ | \cP^n f(x) - \cP^n f(y)| \leq (n+1) \ \cPE_{x^\star} \left[
    |f|(\Phi_n) \un_{n \leq \tau_\alpha} \right] + (n+1) \ \cPE_{y^\star}\left[
    |f|(\Phi_n) \un_{n \leq
      \tau_\alpha} \right] \\
  + (n+1) \ \left| a_x \ast u - a_y \ast u \right| \ast t_{|f|}(n)\eqsp.
\end{multline*}
By (\ref{eq:ControlforAtom}), the first two terms in the right-hand side are
upper-bounded by $C \{V(x) + V(y) \}$.  Applying again
(\ref{eq:ControlforAtom}), $\sup_{k \geq 1} k \ t_{|f|}(k) \leq \sup_{k \geq 1}
k \ \cPE_{\alpha}\left[ W(\Phi_k) \un_{k \leq \tau_\alpha} \right] \leq C
\sup_\Cset V < \infty$, so
\begin{align*}
  (n+1) \ \left| a_x \ast u - a_y \ast u \right| \ast t_{f}(n) &\leq \left(
    \sup_{k \geq 1} k \ t_{W}(k) \right) \quad \sup_{n\geq 1} (n+1) \ 
  \left| a_x \ast u - a_y \ast u \right|(n) \\
  &\leq C \ \sup_{n\geq 1} (n+1) \ \left| a_x \ast u - a_y \ast u \right|(n)
  \eqsp.
\end{align*}
Since $\sup_\Cset V< \infty$, $\cPE_\alpha \left[ \tau_\alpha\right] < \infty$;
standard results from renewal theory imply (see e.g.  \cite{lindvall:1979})
$\sup_{n\geq 1} (n+1) \ \left| a_x \ast u - a_y \ast u \right|(n) \leq C
\{\cPE_{x^\star}\left[ \tau_\alpha \right] + \cPE_{y^\star}\left[ \tau_\alpha
\right] \}$.  The right-hand side is upper bounded by $C \{V(x) + V(y) \}$.
This concludes the proof of (\ref{eq:GoalAtomic}) and thus the overall proof.

{\bf Proof of inequality \eqref{eq:ControlforPm}.} Using $P^m W(x) - W(x) \leq b
\sum_{j=0}^{m-1} P^j(x,\Cset)$,
\begin{align}
\label{eq:RefPourLa}
  (n+1) \ \PE_{x} & \left[ W(\Phi_{n m}) \un_{n \leq \tau_{\mathcal{A}}^{(m)}
    }\right] - W(x) \nonumber \\
 & = \sum_{k=1}^{n} \PE_{x} \left[ (k+1) W(\Phi_{km}) \un_{k \leq
      \tau_{\mathcal{A}}^{(m)}} - k W(\Phi_{(k-1)m}) \un_{k-1 \leq
      \tau_{\mathcal{A}}^{(m)}}
  \right]\nonumber  \\
& \leq \sum_{k=1}^{n} \PE_{x} \left[ W(\Phi_{km}) \un_{k \leq \tau_{\mathcal{A}}^{(m)}}
  \right] + \sum_{k=1}^{n} \PE_{x} \left[ k \{ W(\Phi_{km}) - W(\Phi_{(k-1)m})
    \}
    \un_{k-1 \leq \tau_{\mathcal{A}}^{(m)}} \right] \nonumber  \\
 & \leq \PE_{x} \left[ \sum_{k=1}^{ \tau_{\mathcal{A}}^{(m)}} W(\Phi_{km}) \right] + b \ 
  \PE_{x} \left[ \sum_{k=0}^{\tau_{\mathcal{A}}^{(m)}}(k+1) \ \sum_{j=0}^{m-1}
    P^j(\Phi_{km}, \Cset) \right] \eqsp.
\end{align}
Since ${\mathcal{A}}$ is an accessible set for $P^m$ and $\Dset$ is small for
$P^m$, there exist $c,r>0$ such that $C \un_\Dset(x) \leq
P^{mr}(x,\mathcal{A})$.  Hence, by R\ref{R3}, the first term in
(\ref{eq:RefPourLa}) is upper bounded by $C V(x)$.  $P$ is aperiodic and
$\Cset$ is small: there exist $l \geq 1$ and a non trivial measure $\nu$ such
that $\nu({\mathcal{A}})>0$ and $P^{lm - j}(x,{\mathcal{A}}) \geq \un_\Cset(x)
\nu({\mathcal{A}})$ for any $ j \in \{0, \cdots, m-1\}$ (see e.g.  \cite[proof
of Lemma 14.2.8]{meyn:tweedie:1993}).  Hence, $P^{lm }(x,{\mathcal{A}}) \geq
P^j(x,\Cset) \nu({\mathcal{A}})$ for any $ j \in \{0, \cdots, m-1\}$, and this
yields $ m P^{lm }(x,{\mathcal{A}}) \geq \sum_{j=0}^{m-1} P^j(x,\Cset)
\nu({\mathcal{A}})$.  Therefore, there exists $C < \infty$ such that
\[
\sum_{j=0}^{m-1} P^j(x, \Cset) \leq C \ P^{lm}(x,{\mathcal{A}}) = C \ \PE_x \left[
  \un_{\mathcal{A}}(\Phi_{lm})\right] \eqsp.
\]
This yields
\begin{align*}
  b \ \PE_{x} \left[ \sum_{k=0}^{\tau_{\mathcal{A}}^{(m)}}(k+1) \ 
    \sum_{j=0}^{m-1} P^j(\Phi_{km}, \Cset) \right] &\leq C \ \PE_{x} \left[
    \sum_{k=0}^{\tau_{\mathcal{A}}^{(m)}}(k+1) \ \PE_{\Phi_{km}} \left[
      \un_{\mathcal{A}}(\Phi_{lm})\right] \right]  \\
& \leq C \ \PE_{x} \left[ \sum_{k=l}^{\tau_{\mathcal{A}}^{(m)}+l}(k-l+1) \ \ 
    \un_{\mathcal{A}}(\Phi_{km})\right] \\
 \leq  C \ \PE_{x} \Biggl[ \un_{l \leq \tau_{\mathcal{A}}^{(m)}} \ 
    \sum_{k=\tau_{\mathcal{A}}^{(m)}}^{\tau_{\mathcal{A}}^{(m)}+l} & (k-l+1) \ \ 
    \un_{\mathcal{A}}(\Phi_{km})\Biggr] + C \ \PE_{x} \left[ \un_{l >
      \tau_{\mathcal{A}}^{(m)}} \ \sum_{k=l}^{2l}(k-l+1) \ \ 
    \un_{\mathcal{A}}(\Phi_{km})\right] \\
& \leq C \ \{ \PE_{x} \left[
    \tau_{\mathcal{A}}^{(m)} \right] +1 \} \leq C V(x)\eqsp.
\end{align*}

{\bf Proof of inequality \eqref{eq:ControlforAtom}.} In the sequel, we write
$\tau_\Cset$ as a shorthand notation for $\tau_{\Cset \times \{0,1\}}$, and
$\Phi_{l:n} \notin \Cset$ for $\{\Phi_l \notin \Cset, \cdots, \Phi_n \notin
\Cset \}$. Let $\{\tau^{q}, q \geq 1 \}$ be the successive return times to
$\Cset \times \{0,1\}$.  We write
\begin{multline}
\label{eq:derdesder}
  (k+1) \ \cPE_{x^\star} \left[ W(\Phi_{k}) \un_{k \leq \tau_\alpha }\right] =
  (k+1) \ \cPE_{x^\star} \left[ W(\Phi_{k}) \un_{k \leq \tau_\alpha,
      \tau_\alpha = \tau_\Cset }\right] \\
  + \sum_{q \geq 1} (k+1) \ \cPE_{x^\star} \left[ W(\Phi_{k}) \un_{
      \tau_\Cset^q < k \leq \tau_\alpha} \un_{\Phi_{\tau_\Cset^q+1:k-1} \notin
      \Cset}\right] \eqsp.
\end{multline}
The connection between $P^m$ and $\cP$ yields:
\[
(k+1) \ \cPE_{x^\star} \left[ W(\Phi_{k}) \un_{k \leq \tau_\alpha, \tau_\alpha
    = \tau_\Cset }\right] \leq (k+1) \ \PE_{x} \left[ W(\Phi_{km}) \un_{k \leq
    \tau_\Cset^{(m)} }\right] \leq C V(x) \eqsp.
\]
For the second term, let $q \geq 1$ and consider a general term of the series:
\begin{align}
  (k+1) \ & \cPE_{x^\star} \left[ W(\Phi_{k}) \un_{ \tau_\Cset^q < k \leq
      \tau_\alpha} \un_{\Phi_{\tau_\Cset^q+1:k-1} \notin \Cset}\right]
  \nonumber  \\
  & \leq \sum_{l=0}^{k-1} (k+1-l ) \ \cPE_{x^\star} \left[ \un_{\tau_\Cset^q =
      l} \un_{ l < \tau_\alpha} \cPE_{\Phi_l,
      d_l}\left[W(\Phi_{k-l})\un_{\Phi_{1:k-l-1} \notin \Cset} \right] \right] \label{eq:ControlForAtom-1}\\
  &\qquad\qquad + \sum_{l=0}^{k-1} l \ \cPE_{x^\star} \left[ \un_{\tau_\Cset^q
      = l} \un_{ l < \tau_\alpha} \cPE_{\Phi_l,
      d_l}\left[W(\Phi_{k-l})\un_{\Phi_{1:k-l-1} \notin \Cset} \right] \right]
  \,. \label{eq:ControlForAtom-2}
\end{align}
By definition of the split chain,
\begin{align*}
  (k+1-l )& \cPE_{x^\star} \left[ \un_{\tau_\Cset^q = l} \un_{ l < \tau_\alpha}
    \ \cPE_{\Phi_l, d_l}\left[W(\Phi_{k-l})\un_{\Phi_{1:k-l-1} \notin \Cset}
    \right] \right] \\
&  = \cPE_{x^\star} \left[ \un_{\tau_\Cset^q = l} \un_{ l < \tau_\alpha}
    \cPE_{\Phi_l, d_l}\left[ \cPE_{\Phi_1, d_1}\left[ (k+1-l ) \ 
        W(\Phi_{k-l-1})\un_{\Phi_{0:k-l-2} \notin \Cset} \right] \right]\right] \\
 & \leq \cPE_{x^\star} \left[ \un_{\tau_\Cset^q = l} \un_{ l < \tau_\alpha}
    \cPE_{\Phi_{l}, d_{l}}\left[ \PE_{\Phi_1}\left[ (k+1-l ) \ 
        W(\Phi_{m(k-l-1)})\un_{ \tau_\Cset^{(m)} \geq k-l-1 } \right] \right]\right] \\
 & \leq C \ \cPE_{x^\star} \left[ \un_{\tau_\Cset^q = l} \un_{ l < \tau_\alpha} \cPE_{\Phi_{l}, d_{l}}\left[ V(\Phi_1)\right] \right] \leq \sup_\Cset RV \  \cPP_{x^\star} \left( \un_{\tau_\Cset^q = l} \un_{ l < \tau_\alpha} \right)\eqsp.
\end{align*}
Hence, (\ref{eq:ControlForAtom-1}) is upper bounded by $C \ \cPE_{x^\star}
\left( \tau_\Cset^q < \tau_\alpha \right)$ and thus by $C \ 
(1-\epsilon)^{q-1}$.  Furthermore, (\ref{eq:ControlForAtom-2}) is upper bounded
by
\[
C \ \sup_\Cset R W \ \sum_{l=0}^{k-1} l \ \cPP_{x^\star} \left( \tau_\Cset^q =
  l, l < \tau_\alpha\right) \leq C \cPE_{x^\star}\left[ \tau_\Cset^q \ 
  \un_{\tau_\Cset^q < \tau_\alpha}\right] \eqsp.
\]
The decomposition $\tau_\Cset^q = \sum_{r=1}^{q-1} \{ \tau_\Cset^{r+1} -
\tau_\Cset^{r} \} + \tau_\Cset$, and the inequalities $\cPP_{x^\star}(
\tau_\Cset^r < \tau_\alpha) \leq (1-\epsilon)^r$ and $\sup_{\Cset \times
  \{0,1\}} \cPE_{x,d}\left[ \tau_\Cset \right] < \infty$, imply that $
\cPE_{x^\star}\left[ \tau_\Cset^q \ \un_{\tau_\Cset^q < \tau_\alpha}\right]
\leq C (1-\epsilon)^q \ \cPE_{x^\star}[\tau_\Cset]$.

The second term in the rhs of \eqref{eq:derdesder} is a convergent series. This
concludes the proof.


\subsection{Proof of Proposition~\ref{lemma:Moment2Ergod}}
Under the stated assumptions on $r$, $\inf_\Xset V >0$, and $\inf_\Xset W >0$.
For any measurable set $\Cset$ and any function $f : \Xset \to [1,\infty)$ the
function $F(x) \eqdef \PE_x \left[ \sum_{k=0}^{\sigma_\Cset} f(\Phi_k) \right]$
satisfies
\[
PF(x) = \PE_x \left[ \sum_{k=1}^{\tau_\Cset} f(\Phi_k) \right] = F(x) - f(x) +
\un_\Cset(x) \ \PE_x \left[ \sum_{k=1}^{\tau_\Cset} f(\Phi_k) \right] \leq F(x)
-f(x) + b \un_\Cset(x)
\]
where $b \eqdef \sup_{x \in \Cset} \PE_x \left[ \sum_{k=1}^{\tau_\Cset}
  f(\Phi_k) \right]$.

\subsection{Proof of Proposition~\ref{Drift:sequence}}
\label{sec:ProofDrift:sequence}
Upon noting that $r,f$ are non-negative and that $\tau^{k+1} = \tau^k + \tau
\circ \theta^{\tau^k}$ $\PP_x$-\as
  \[
    \PE_x \left[ \sum_{k=0}^{\tau_\Cset -1} r(k) \ f(\Phi_k) \right] \leq \PE_x
    \left[ \sum_{k=0}^{\tau^{\bar \tau_\Cset} -1} r(k) \ f(\Phi_k) \right] \leq
    \PE_x \left[ \sum_{k=0}^{\bar \tau_\Cset -1} \sum_{j= 0}^{\tau \circ
        \theta^{\tau^k}-1} \ r(j+ \tau^k) \ f(\Phi_{j+\tau^k}) \right] \eqsp.
  \]
  By definition of the random time $\tau$, $\tau \circ \theta^{\tau^k} =
  n(\Phi_{\tau^k})$ and this implies
  \begin{align*}
    \PE_x \left[ \sum_{k=0}^{\tau_\Cset -1} r(k) \ f(\Phi_k) \right] & \leq
    \PE_x \left[ \sum_{k=0}^{\bar \tau_\Cset -1} \sum_{j=
        0}^{n(\Phi_{\tau^k})-1} \ r(j+ \tau^k) \ f(\Phi_{j+\tau^k}) \right] \\
    & = \PE_x \left[ \sum_{k=0}^{\bar \tau_\Cset -1} \sum_{l \geq k}
      \un_{\tau^k = l} \PE_{\Phi_l} \left[ \sum_{j= 0}^{n(\Phi_{0})-1} \ r(j+
        l) \ f(\Phi_{j}) \right]\right].
  \end{align*}
  By the drift assumption, $\PP_x$-\as,
\[
\PE_{\Phi_l} \left[ \sum_{j= 0}^{n(\Phi_{0})-1} \ r(j+ l) \ f(\Phi_{j}) \right]
\leq V_l(\Phi_l) - \PE_{\Phi_l} \left[ V_{l+n(\Phi_0)}(\Phi_{n(\Phi_0)})\right]
+ S_l(\Phi_l) \un_\Cset(\Phi_l) \eqsp,
\]
so that
\begin{multline*}
  \PE_x \left[ \sum_{k=0}^{\tau_\Cset -1} r(k) \ f(\Phi_k) \right] \leq \PE_x
  \left[ \sum_{k=0}^{\bar \tau_\Cset -1} \left\{ V_{\tau^k}(\Phi_{\tau^k}) -
      V_{{\tau^k}+n(\Phi_{\tau^k})}(\Phi_{\tau^k+n(\Phi_{\tau^k})}) +
      S_{\tau^k}(\Phi_{\tau^k}) \un_\Cset(\Phi_{\tau^k}) \right\} \right] \\
  \leq \PE_x \left[ \sum_{k=0}^{\bar \tau_\Cset -1} \left\{
      V_{\tau^k}(\Phi_{\tau^k}) - V_{\tau^{k+1}}(\Phi_{\tau^{k+1}}) +
      S_{\tau^k}(\Phi_{\tau^k}) \un_\Cset(\Phi_{\tau^k}) \right\} \right] \leq
  V_0(x) + S_0(x) \un_\Cset(x) \eqsp.
\end{multline*}

\subsection{Proof of Theorem~\ref{theo:subsampled2moments}}
\label{proof:theo:subsampled2moments}
Define $\tau \eqdef n(\Phi_0)$ and the iterates $\tau^1 \eqdef \tau $,
$\tau^{k+1} \eqdef \tau \circ \theta^{\tau^k} + \tau^k$ for $k \geq 1$, where
$\theta$ denotes the shift operator. (By convention, $\tau^0 = 0$.) Finally,
set $\bar \tau_\Cset \eqdef \inf \{k \geq 1, \Phi_{\tau^k} \in \Cset \}$.

{\em Proof of Theorem~\ref{theo:subsampled2moments}(\ref{case1}).} By
definition of the random time $\tau$, $\PP_x$-\as, $ \tau_\Cset \leq \tau^{\bar
  \tau_\Cset} = \sum_{k=0}^{\bar \tau_\Cset -1} n\left(\Phi_{\tau^k} \right)$.
Since $t \mapsto R(t)/t$ is non-increasing, we have $R(a+b) \leq R(a) +R(b)$
for any $a,b \geq 0$.  This property, combined with the fact that $R$ is
increasing yields $\PP_x$-\as
 \[ 
 R\left( \tau_\Cset \right) \leq \sum_{l \geq 1} \un_{\bar \tau_\Cset = l} \ 
 R\left( \sum_{k=0}^{l -1} n\left(\Phi_{\tau^k} \right) \right) \leq \sum_{l
   \geq 1} \un_{\bar \tau_\Cset = l} \ \sum_{k=0}^{l -1} R \circ
 n\left(\Phi_{\tau^k} \right) \leq \sum_{k=0}^{\bar \tau_\Cset -1}
 W\left(\Phi_{\tau^k} \right) \eqsp,
 \]
 where we used $R(n(x)) \leq W(x)$ in the last inequality.  The proof is
 concluded upon noting that from the drift assumption and the Comparison
 Theorem~\cite[Proposition 11.3.2]{meyn:tweedie:1993},
 \[
 (1-\beta) \ \PE_x \left[ \sum_{k=0}^{\bar \tau_\Cset-1} W(\Phi_{\tau^k})\right]
 \leq W(x) + b \un_\Cset(x) \eqsp.
 \]
 
 {\em Proof of Theorem~\ref{theo:subsampled2moments}(\ref{case2}).} Under the
 stated assumption, the inverse $H \eqdef R^{-1}$ exists. Set $r(t) \eqdef
 R'(t) = 1/(H' \circ R(t))$. Define a sequence of measurable functions $\{H_k,
 k \in \nset \}$, $H_k : [1, \infty) \to (0, \infty)$ by $H_k(t) \eqdef
 \int_0^{H(t)} r(z+k) \ dz = R(H(t)+k)-R(k)$.  Then $H_k$ is increasing and
 concave. Indeed $H_k'(t) =\frac{r(H(t)+k)}{r(H(t))}$ and this is positive
 since $R$ is increasing. Since $R$ (and thus $H$) is increasing and t
 $t\mapsto r(t+k)/r(t)$ is non-increasing (since $R'$ is log-concave), $H_k'$
 is non-increasing.  The Jensen's inequality and the drift assumption imply for
 any $k \in \nset$, $x \in \Xset$,
\begin{align*}
  P(H_{k+n(x)}\circ W(\Phi_{n(x)})) & \leq H_{k+n(x)}\left(
    P W(\Phi_{n(x)}) \right) \leq H_{k+n(x)}\left(
    \beta W(x)  + b\un_\Cset(x)\right) \\
  & \leq H_{k+n(x)}(\beta W(x)) + H_{k+n(x)}(b) \un_\Cset(x)  \\
  & \leq H_{k} (W(x)) - \sum_{j=0}^{n(x)-1} r(k+j) +H_{k+n(x)}(b)\un_\Cset(x)
  + \mathcal{R}_k(x) \eqsp,
\end{align*}
where we defined $\mathcal{R}_k(x) \eqdef H_{k+n(x)}(\beta W(x)) -H_k(W(x)) +
\sum_{j=0}^{n(x)-1} r(k+j) $.  We now prove that $\mathcal{R}_k(x) \leq 0$
which will conclude the proof by applying Proposition~\ref{Drift:sequence} with
$V_k=H_k\circ W$. We have
\begin{align*}
  \mathcal{R}_k(x)& =  \sum_{j=0}^{n(x)-1}r(j+k) + \int_{n(x)}^{H(\beta  W(x))+n(x)}  r(z + k) dz - \int_{0}^{H(W(x))} r(z + k) dz  \\
  & \leq \sum_{j=0}^{n(x)-1}r(j+k) - \int_0^{n(x)} r(z + k) dz \eqsp,
\end{align*}
since by assumption, $H(W) \geq H(\beta W) + n$.  Now, $R$ is convex and so $r$
is non-decreasing. Therefore, $ \sum_{j=0}^{n(x)-1}r(j+k) \leq \int_0^{n(x)}
r(z+k) dz$ and this concludes the proof.

\subsection{Proof of Corollary~\ref{coro:Pi-integ}}
\label{proof:coro:Pi-integ}
\textit{(a)} This follows from \cite[Lemma 3.1 ]{nummelin:tuominen:1983} (see
also \cite[Proposition 3.1]{tuominen:tweedie:1994}).  \textit{(b)} By
\cite[Theorem 14.2.11]{meyn:tweedie:1993}, there exist measurable sets $\{A_n,
n\geq 0 \}$ whose union is full and such that $\sup_{x \in A_n} \PE_x
\left[\sum_{k=0}^{\tau_\Dset-1} W(\Phi_k) \right] < \infty$ for any accessible
set $\Dset$. For an accessible small set $\Dset$, there exists $n$ such that
$\tilde \Dset \eqdef \Dset \cap A_n$ is small and accessible.  The proof
follows by combining the results of \textit{(a)} with $\sup_{x \in \tilde
  \Dset} \PE_x \left[\sum_{k=0}^{\tau_{\tilde \Dset}-1} W(\Phi_k) \right] <
\infty$.

\bibliographystyle{plain} \bibliography{Warwick}

 \end{document}